\documentclass{amsart}

\usepackage{graphicx}
\usepackage{tikz}
\usetikzlibrary{matrix,arrows,decorations.pathmorphing}
\usepackage{tikz-cd}
\usepackage{hyperref}
\usepackage{amsthm}
\usepackage[mathscr]{euscript}
\usepackage{amssymb}
\usepackage{mathrsfs}
\usepackage{enumitem}
\newtheorem{theorem}{Theorem}[section]
\newtheorem{lemma}[theorem]{Lemma}
\newtheorem*{conjecture*}{Conjecture}
\newtheorem*{convention*}{Convention}

\theoremstyle{definition}

\newtheorem{corollary}[theorem]{Corollary}

\theoremstyle{remark}
\newtheorem{remark}[theorem]{Remark}
\newtheorem{remarkthm}{Remark}[theorem]

\numberwithin{equation}{section}

\usepackage{color}

\usepackage[all,cmtip]{xy}

\newcommand{\Z}{\mathbb{Z}}

\newcommand{\opn}{\operatorname}

\newcommand{\BP}{\opn{BP}}

\newcommand{\CH}{\opn{CH}}

\begin{document}
\title[On $\BP^*(BPU_n)$ in lower dimensions and the Thom map]{On the Brown-Peterson cohomology of $BPU_n$ in lower dimensions and the Thom map}



\author{Xing Gu}
\address{Max Planck Institute for Mathematics, Vivatsgasse 7, 53111 Bonn, Germany}
\email{gux2006@mpim-bonn.mpg.de}
\thanks{The author would like to thank the Max Planck Institute for Mathematics for their hospitality and financial support.}

\subjclass[2010]{55N35, 55R35}

\date{}

\dedicatory{}

\keywords{the Brown-Peterson cohomology, the classifying spaces of the projective unitary groups}

\begin{abstract}
For an odd prime $p$, we study the image of the Thom map from Brown-Peterson cohomology of $BPU_n$ to the ordinary cohomology in dimensions $0\leq i\leq 2p+2$, where $BPU_n$ is the classifying space of the projective unitary group $PU_n$. Also we show that a family of well understood $p$-torsion cohomology classes $y_{p,k}\in H^{2p^{k+1}+2}(BPU_n;\Z_{(p)})$ are in the image of the Thom map.
\end{abstract}

\maketitle
\section{Introduction}\label{sec:intro}
Let $p$ be an odd prime number, and let $\BP$ be the corresponding Brown-Peterson spectrum. The Brown-Peterson cohomology $\BP^*(BG)$ of the classifying space of a compact Lie group or a finite group $G$  is the subject of various works such as Kameko and Yagita \cite{kameko2008brown}, Kono and Yagita \cite{kono1993brown}, Leary and Yagita \cite{leary1992some}, and Yan \cite{YAN1995221}.

One case that $\BP^*(BG)$ is particularly interesting is when $G$ is homotopy equivalent to a complex algebraic group via a group homomorphism. In this case the Chow ring of $BG$, $\CH^*(BG)$ is defined by Totaro \cite{totaro1999chow}, and one has the cycle class map
\begin{equation}\label{eq:cl}
\opn{cl}:\CH^*(BG)\rightarrow H^{\textrm{even}}(BG;\Z)
\end{equation}
which is a ring homomorphism from the Chow ring to the subring of $H^*(BG)$ of even dimensional classes. Although for complex algebraic varieties, Chow rings are in general much more complicated than ordinary cohomology, it is shown in many cases that $\CH^*(BG)$ is simpler than $H^*(BG;\Z)$. On the other hand, Totaro \cite{totaro1999chow} shows that the cycle class map \eqref{eq:cl} factors as
\begin{equation}\label{eq:cl refined}
\CH^*(BG)\xrightarrow{\tilde{\opn{cl}}}\opn{MU}^{\textrm{even}}(BG)\otimes_{\opn{MU}^*}\Z\xrightarrow{T}
H^{\textrm{even}}(BG;\Z)
\end{equation}
where $\opn{MU}$ denotes the complex cobordism theory, and the second map $T$ is the Thom map. The first map $\tilde{\opn{cl}}$ is called the refined cycle class map. Therefore, the $\BP$ theory, being a $p$-local approximation of the $\opn{MU}$ theory, act as a bridge between the Chow ring and the ordinary cohomology of $BG$. Indeed, it is an interesting problem to find out for which $G$ is the refined cycle class map
\[\tilde{\opn{cl}}: \CH^*(BG)\xrightarrow{\tilde{\opn{cl}}}\opn{MU}^*(BG)\otimes_{\opn{MU}^*}\Z\]
an isomorphism. For this to hold, it is necessary that $\BP^*(BG)$ concentrates in even dimensions. This property is studied for various $G$ by Landweber \cite{landweber1970coherence} and \cite{landweber1972elements}, and by Kono and Yagita \cite{kono1993brown}.

In this paper we focus on the case $G=PU_n$, where $PU_n$ is the $n$th projective unitary group, i.e., the quotient of the unitary group $U_n$ by its center $S^1$, the group of the unit circle, or equivalently, the quotient of the special unitary group $SU_n$ by its center, the group of $n$th complex roots of unit.

The algebraic invariants of $BPU_n$ are much less known compared to $BG$ for most of the other compact Lie groups $G$. The Chow ring of $BPU_3$ is determined, up to one relation, by Vezzosi \cite{vezzosi1999chow}. The additive structure of $\CH^*(BPU_3)$ is independently determined by Kameko and Yagita \cite{kameko2008brown}. Vistoli \cite{vistoli2007cohomology} improves Vezzosi's method and determines the additive structures as well as much of the ring structures of the Chow ring and ordinary cohomology with integral coefficients of $BPU_p$ for an odd prime $p$. In particular, he completes Vezzosi's study of the Chow ring of $BPU_3$. The ordinary mod $p$ cohomology and Brown-Peterson cohomology of $BPU_p$ are studied by Kameko and Yagita \cite{kameko2008brown}, Kono and Yagita \cite{kono1993brown}, and Vavpeti{\v{c}} and Viruel \cite{vavpetivc2005mod}. The mod $2$ ordinary cohomology ring of $BPU_n$ for $n\equiv 2\pmod{4}$ is determined by Kono and Mimura \cite{kono1975cohomology} and Toda \cite{toda1987cohomology}.

For a general positive integer $n$, the cohomology groups $H^k(BPU_n;\Z)$ for $k\leq 3$ are easily determined by the universal $n$-cover $SU_n\to PU_n$. The group $H^4(BPU_n;\Z)$ is determined by Woodward \cite{woodward1982classification} and $H^5(BPU_n;\Z)$ by Antieau and Williams \cite{antieau2014topological}. The ring structure of $H^*(BPU_n;\Z)$ in dimensions less than or equal to $10$ is determined by the author \cite{gu2019cohomology}. The author \cite{gu2019some} also studies some $p$-torsion classes of $\CH^*(BPU_n)$ for $n$ with $p$-adic valuation $1$, i.e., $p|n$ but $p^2\nmid n$. To the author's best knowledge, $\BP^*(BPU_n)$ for $n$ not a prime number has not been studied in any earlier published work.

Before stating the main conclusions of this paper, we fix some notations. For a spectrum $A$, we denote by $A_*$ its homotopy groups, or the group of coefficients of the homology theory $A$, considered as a graded abelian group. Denote by $A^*$ the group of coefficients of the cohomology theory associated to $A$. Then $A^*$ and $A_*$ are isomorphic, but the gradings are opposite to each other. For instance, we have $\BP_*\cong\Z_{(p)}[v_1,v_2\cdots]$ where $\opn{dim}v_k=2p^k-2$ and $\BP^*\cong\Z_{(p)}[v_1,v_2\cdots]$ where $\opn{dim}v_k=-(2p^k-2)$.

Let $H\Z_{(p)}$ be the Eilenberg-Mac Lane spectrum for the ring $\Z_{(p)}$, and $T:\BP\to H\Z_{(p)}$ be the Thom map. Then we have the augmentation map $T:\BP^*\to\Z_{(p)}$ induced by the Thom map, and the ring $\Z_{(p)}$ has a structure of graded $\BP^*$-algebra given by $T^*$.

More generally, for any space $X$, we have the induced homomorphism $T:\BP^*(X)\to H^*(X;\Z_{(p)})$, whose image is canonically isomorphic to the graded $\Z_{(p)}$-algebra $\BP^*(X)\otimes_{BP^*}\Z_{(p)}$.
\begin{theorem}\label{thm:main}
Let $p$ be an odd prime. In dimensions $0\leq k \leq 2(p+1)$, the graded ring $\BP^*(BPU_n)\otimes_{BP^*}\Z_{(p)}$, i.e., the image of the Thom map
\[T:\BP^*(BPU_n)\to H^*(BPU_n;\Z_{(p)})\]
concentrates in even dimensions.
\end{theorem}
\begin{remarkthm}
We have the ``canonical Brauer class'' denoted by $x_1$ generating the group
\[H^3(BPU_n,\Z_{(p)})\cong\Z_{p^r},\]
where $r$ is the $p$-adic valuation of $n$. In other words, we have $n=p^rm$ for with $p\nmid m$. This class plays an important role in the calculation of $H^*(BPU_n;\Z)$ in \cite{gu2019cohomology}. However, Theorem \ref{thm:main} indicates that any $kx_1$ for $p^r\nmid k$ is not in the image of the Thom map.
\end{remarkthm}
For the next theorem, we note that there are $p$-torsion classes
\[y_{p,k}\in H^{2p^{k+1}+2}(BPU_n;\Z_{(p)}),\ k\geq0,\]
which are studied in \cite{gu2019cohomology} and discussed in more details in Section \ref{sec:ordinary coh}.
\begin{theorem}\label{thm:eta p,k}
Let $p$ be an odd prime. For $k\geq 0$ and $p|n$, there are classes
\begin{equation*}
\eta_{p,k}\in\BP^{2p^{k+1}+2}(BPU_n)
\end{equation*}
satisfying
\[T(\eta_{p,k})=y_{p,k}\in H^{2p^{k+1}+2}(BPU_n;\Z_{(p)})\]
where $T$ is the Thom map.
\end{theorem}
\begin{remarkthm}
Localizing at $p$ the homotopy fiber sequence $B\Z_n\rightarrow BSU_n\rightarrow BPU_n$, we obtain a $p$-local homotopy equivalence $BSU_n\xrightarrow{\simeq_{(p)}} BPU_n$ in the case $p\nmid n$. Therefore this case is not very interesting.
\end{remarkthm}
\begin{remarkthm}
In \cite{gu2019some}, the author constructs $p$-torsion classes $\rho_{p,k}$, $k\geq 0$ in the Chow ring of $BPU_n$ satisfying $\opn{cl}(\rho_{p,k})=y_{p,k}$, when the $p$-adic valuation of $n$ is $1$. Theorem \ref{thm:eta p,k} and the existence of the refined cycle class map as in \eqref{eq:cl refined} support the conjecture that the classes $\rho_{p,k}$ exist for $n$ with $p$-adic valuation greater than $1$.
\end{remarkthm}
\begin{remarkthm}
The classes $\eta_{p,k}$ are not $p$-torsion classes in general. As pointed out by N. Yagita, it follows from Theorem 1.4 of Kameko and Yagita \cite{kameko2008brown}, for $n=p$, that the class $\eta_{p,0}$ satisfies
\[p\eta_{p,0}=v_2\eta_{p,0}^p+\cdots\neq 0,\]
where the class on the right side is trivial modulo the ideal $(v_2,v_3,\cdots,v_n,\cdots)$ of the ring $\BP^*$.
\end{remarkthm}
This paper is organized as follows. In Section \ref{sec:ordinary coh} we review some results on the ordinary cohomology of $BPU_n$, most of which are proved in \cite{gu2019cohomology} and \cite{gu2019some}. In Section \ref{sec:p-primary} we prove Theorem \ref{thm:main} by studying an $n$-torsion abelian subgroup of $PU_n$, an idea inspired by Vistoli \cite{vistoli2007cohomology}. Finally, in Section \ref{sec:eta}, we construct the classes $\eta_{p,k}$ and prove Theorem \ref{thm:eta p,k}.
\subsection*{Acknowledgement}
The author thanks B. Totaro and N. Yagita for pointing out several errors in an earlier version. This paper is written during a visit to the Max Planck Institute for Mathematics (MPIM), in the midst of the COVID-19 pandemic. The author would like to thank MPIM for their supports in various ways during this difficult time.

\section{On the ordinary cohomology of $BPU_n$}\label{sec:ordinary coh}
In this section we consider the ordinary cohomology of $BPU_n$. For the most part of this section, we reformulate results in \cite{gu2019cohomology}. By definition we have a short exact sequence of Lie groups
\[1\rightarrow\Z_n\rightarrow SU_n\rightarrow PU_n\rightarrow 1\]
which gives a universal cover of $PU_n$ and shows $\pi_1(PU_n)\cong\Z_n$. It follows from the Hurewicz theorem and the universal coefficient theorem that we have
\begin{equation}\label{eq:H123}
H^k(BPU_n;\Z)\cong
\begin{cases}
0,\hspace{2 mm}k=1,2,\\
\Z_n,\hspace{2 mm}k=3.
\end{cases}
\end{equation}
Consider the short exact sequence of Lie groups
\[1\rightarrow S^1\rightarrow U_n\rightarrow PU_n\rightarrow 1\]
which defines the Lie group $PU_n$. Taking classifying spaces, we obtain a homotopy fiber sequence
\[BS^1\rightarrow BU_n\rightarrow BPU_n.\]
Notice that $BS^1$ is of the homotopy type of the Eilenberg-Mac Lane space $K(\Z,2)$. Delooping $BS^1$, we obtain another homotopy fiber sequence
\begin{equation}\label{eq:BPUn fib seq}
BU_n\rightarrow BPU_n\xrightarrow{\chi}K(\Z,3).
\end{equation}
Here the map
\begin{equation}\label{eq:chi}
\chi: BPU_n\rightarrow K(\Z,3)
\end{equation}
defines a generator $x_1$ of $H^3(BPU_n;\Z)\cong\Z_n$, or $H^3(BPU_n;\Z_{(p)})\cong\Z_{p^r}$, where $r$ is the $p$-adic valuation of $n$, i.e., we have $n=p^rm$ with $p\nmid m$. We call $x_1$ the canonical Brauer class of $BPU_n$.

In principle, the homology of $K(\pi,n)$ for any finitely generated abelian group $\pi$ and any $n>0$ is determined in \cite{Ca}. Tamanoi \cite{tamanoi1999subalgebras} offers a description of the mod $p$ cohomology
of $K(\pi,n)$ in terms of the Milnor basis (\cite{milnor1958steenrod}) of the mod $p$ Steenrod algebra. In \cite{gu2019cohomology}, the author gives a description of the cohomology of $K(\Z,3)$, which is consistent with the notations in this paper. 

Throughout the rest of this paper, we denote by $\mathscr{P}^k$ the $k$th Steenrod reduced power operation, and $\delta$ the Bockstein homomorphism $H^*(-;\Z_p)\rightarrow H^{*+1}(-;\Z)$.

The following lemma is well known and can be easily deduced from Section 2 of \cite{gu2019cohomology}.
\begin{lemma}\label{lem:K(Z,3)}
In dimensions $0<i\leq 2p+4$, we have
\begin{equation*}
H^i(K(\Z,3);\Z_{(p)})\cong
\begin{cases}
\Z_{(p)},\ i=3,\\
\Z_p,\ i=2p+2,\\
0,\ 0<i\leq 2p+4,\ i\neq3,\ 2p+2.
\end{cases}
\end{equation*}
Let $x_1\in H^3(K(\Z,3);\Z_{(p)})$ be the canonical Brauer class. Then the group
\[H^{2p+2}(K(\Z,3);\Z_{(p)})\cong\Z_p\]
is generated by the class $y_{p,0}=\delta\mathscr{P}^1(x_1)$. The mod $p$ reduction of $y_{p,0}$, denoted by $\bar{y}_{p,0}$, is equal to $Q_1(x_1)$, where $Q_1$ is one of the Milnor's operations defined in \cite{milnor1958steenrod}.
\end{lemma}

We also note the following
\begin{lemma}\label{lem:K(Z,3)p-tor}
All torsion classes in the graded abelian group $H^*(K(\Z,3);\Z_{(p)})$ are $p$-torsion classes. In other words, the abelian groups $H^k(K(\Z,3);\Z_{(p)})$ are $p$-torsion groups for $k>3$.
\end{lemma}
\begin{proof}
This follows immediately from Proposition 2.14 of \cite{gu2019cohomology}, which gives a complete description of the graded $\Z_{(p)}$-algebra $H^*(K(\Z,3);\Z_{(p)})$.
\end{proof}
\begin{corollary}\label{cor:K(Z,3)p-tor}
For $k>3$, the mod $p$ reduction
\[H^k(K(\Z,3);\Z_{(p)})\rightarrow H^k(K(\Z,3);\Z_{p})\]
is injective.
\end{corollary}
\begin{proof}
This follows immediately from Lemma \ref{lem:K(Z,3)p-tor} and the long exact sequence induced by the short exact sequence
\[0\rightarrow\Z_{(p)}\xrightarrow{\times p}\Z_{(p)}\rightarrow\Z_p\rightarrow 0.\]
\end{proof}
In the cohomology ring $H^*(K(\Z,3);\Z_{(p)})$ we have $p$-torsion classes
\[y_{p,k}=\delta\mathscr{P}^{p^k}\mathscr{P}^{p^{k-1}}\cdots\mathscr{P}^1(x_1),\hspace{2 mm}k\geq 0\]
of dimension $2p^{k+1}+2$, where $x_1\in H^3(BPU_n;\Z)$ is the canonical Brauer class. Let $\bar{y}_{p,k}\in H^*(K(\Z,3);\Z_p)$ be the mod $p$ reduction of $y_{p,k}$, and we have $\bar{y}_{p,k}=Q_{k+1}(x_1)$ where $Q_{k+1}$ is the Milnor's operation considered in \cite{milnor1958steenrod}.

In \cite{gu2019cohomology} and \cite{gu2019some}, the author studies the images of the classes $y_{p,k}$ under
\[\chi^*: H^*(K(\Z,3);\Z_{(p)})\rightarrow H^*(BPU_n;\Z_{(p)}).\]
In the following theorems, we abuse notations and denote $\chi^*(y_{p,k})$ simply by $y_{p,k}$.
\begin{theorem}[Theorem 1.2, \cite{gu2019cohomology}]\label{thm:2p+2}
Let $p$ be a prime. In $H^{2p+2}(BPU_{n};\mathbb{Z})$, we have $y_{p,0}\neq 0$ of order $p$ when $p|n$, and $y_{p,0}=0$ otherwise. Furthermore, the $p$-torsion subgroup of $H^k(BPU_n;\mathbb{Z})$ is $0$ for $3<k<2p+2$.
\end{theorem}
\begin{theorem}[(1) of Theorem 1.1, \cite{gu2019some}]\label{thm:2p^k+2}
In $H^{2p^{k+1}+2}(BPU_n;\Z_{(p)})$, we have $p$-torsion classes $y_{p,k}\neq 0$ for all odd prime divisors $p$ of $n$ and $k\geq 0$.
\end{theorem}

\section{An $n$-torsion abelian subgroup of $PU_n$}\label{sec:p-primary}
Let $p$ be an odd prime. In \cite{vistoli2007cohomology}, Vistoli considers an abelian $p$-subgroup of $PU_p$, which plays an important role in the study of the Chow ring and cohomology of $BPU_p$. In this section we slightly generalize his construction and prove Theorem \ref{thm:main}.

Let $n>1$ be an integer. Consider the following matrices in $U_n$:
\begin{equation*}
\alpha'=
\begin{bmatrix}
0 & 1\\
I_{n-1} & 0
\end{bmatrix}\\
\textrm{ and }\\
\beta'=
\begin{bmatrix}
\zeta & & & & \\
 & \zeta^2& & &\\
& & \ddots & &\\
& & & \zeta^{n-1} & \\
& & & &1
\end{bmatrix}
\end{equation*}
where $\zeta=\exp(2\pi i/n)$. Let $V_n'$ be the subgroup of $U_n$ generated by $\alpha'$ and $\beta'$. A direct calculation shows
\begin{equation}\label{eq:alpha'beta'}
\beta'\alpha'=\zeta\alpha'\beta'.
\end{equation}
Let $W_n$ be the subgroup of $U_n$ generated by the matrix $\beta'$ and the scalar $\zeta$. Then we have $W_n\cong\Z_n\times\Z_n$, which is a normal subgroup of $V_n'$, and we have the quotient group $V_n'/W_n\cong\Z_n$ generated by the matrix $\alpha'$. Indeed, $V_n'$ is a semidirect product
\begin{equation}\label{eq:V'SES}
V_n'=(\Z_n\times\Z_n)\rtimes_{\phi} C_n
\end{equation}
where $C_n$ is the cyclic group of order $n$. Here we use a different notation than $\Z_n$ since it looks less confusing as we introduce the action $\phi$.
\begin{lemma}\label{lem:phi}
In the definition \eqref{eq:V'SES} of $V_n'$, the action of $\phi$ is as follows: Identify the canonical generator of $C_n$ with the matrix $\alpha'$. Then $\alpha'$ acts on $\Z_n\times\Z_n$ as the matrix
\begin{equation*}
\begin{bmatrix}
1 & -1\\
0 & 1
\end{bmatrix}.
\end{equation*}
\end{lemma}
\begin{proof}
The  action $\phi$ of $\alpha'$ is given by the conjugation action of $\alpha'$ on $W_n$, and a direct computation shows
\begin{equation*}
\alpha'[\zeta^i(\beta')^j](\alpha')^{-1}=\zeta^{i-j}(\beta')^j,
\end{equation*}
and we conclude.
\end{proof}
\begin{lemma}\label{lem:V'}
There is a short exact sequence
\[0\rightarrow\Z_n\rightarrow H^2(BV_n';\Z)\rightarrow\Z_n\rightarrow 0.\]
\end{lemma}
\begin{remarkthm}
Later in the proof of Lemma \ref{lem:H3 restriction}, we may show 
\[H^2(BV_n';\Z)\cong\Z_n\oplus\Z_n.\]
But this does not seem important in the rest of this paper.
\end{remarkthm}
\begin{proof}
By \eqref{eq:V'SES} we have a short exact sequence
\[1\rightarrow\Z_n\times\Z_n\rightarrow V_n'\rightarrow C_n\rightarrow1.\]
Consider the Lyndon-Hochschild-Serre spectral sequence
\[E_2^{s,t}\cong H^s(BC_n; H^t(B\Z_n\times B\Z_n;\Z)_{\phi})\Rightarrow H^{s+t}(BV_n';\Z),\]
where $H^t(B\Z_n\times B\Z_n;\Z)_{\phi}$ means the local coefficient system induced by $\phi$.

The only nontrivial groups $E_2^{s,t}$ with $s+t=2$ are $E_2^{2,0}$ and $E_2^{0,2}$. For obvious degree reasons there is no nontrivial differential into or out of either of them. Therefore we have $E_2^{2,0}=E_{\infty}^{2,0}$ and $E_2^{0,2}=E_{\infty}^{0,2}$, and a short exact sequence
\begin{equation}\label{eq:H2SES}
0\rightarrow E_2^{2,0}\rightarrow H^2(BV_n';\Z)\rightarrow E_2^{0,2}\rightarrow 0.
\end{equation}
The local coefficient system on the bottom row of the spectral sequence is the constant one, and we have
\[E_2^{2,0}\cong H^2(BC_n; H^0(B\Z_n\times B\Z_n;\Z))=H^2(BC_n;\Z)\cong\Z_n.\]
For $E_2^{0,2}$, we have
\[E_2^{0,2}=H^0(BC_n; H^2(B\Z_n\times B\Z_n;\Z)_{\phi})=H^2(B\Z_n\times B\Z_n;\Z)^{\phi}\cong (\Z_n\times\Z_n)^{\phi},\]
i.e., the invariants of $\Z_n\times\Z_n$ under the action $\phi$. By Lemma \ref{lem:phi}, we have
\[E_2^{0,2}\cong (\Z_n\times\Z_n)^{\phi}\cong\Z_n.\]
By \eqref{eq:H2SES}, we conclude.
\end{proof}

Let $\alpha$ and $\beta$ be the conjugation class in $PU_n$ of $\alpha'$ and $\beta'$, respectively. By \eqref{eq:alpha'beta'} we have $\alpha\beta=\beta\alpha$. Let $V_n$ be the subgroup of $PU_n$ generated by $\alpha$ and $\beta$, and we have $V_n\cong\Z_n\times\Z_n$.
\begin{remark}
In the case $n=p$, the subgroup $V_p\subset PU_p$ plays an important role in Vistoli \cite{vistoli2007cohomology}, where the Chow ring and integral cohomology of $BPU_p$ are thoroughly studied.
\end{remark}
\begin{lemma}\label{lem:H3}
$H^2(BV_n;Z)\cong\Z_n\oplus\Z_n$, $H^3(BV_n;\Z)\cong\Z_n.$
\end{lemma}
\begin{proof}
This follows from the isomorphism $V_n\cong\Z_n\times\Z_n$ and the K{\"u}nneth formula.
\end{proof}
Recall that by \eqref{eq:H123} we have $H^3(BPU_n;\Z)\cong\Z_n$.
\begin{lemma}\label{lem:H3 restriction}
The inclusion $V_n\subset PU_n$ induces an isomorphism
\[H^3(BPU_n;\Z)\xrightarrow{\cong} H^3(BV_n;\Z).\]
\end{lemma}
\begin{proof}
By construction we have a short exact sequence of groups
\[1\rightarrow\Z_n\rightarrow V'_n\rightarrow V_n\rightarrow 1,\]
which induces a homotopy fiber sequence
\[B\Z_n\rightarrow BV'_n\rightarrow BV_n.\]
Delooping the first term, we obtain another homotopy fiber sequence
\begin{equation*}
BV'_n\rightarrow BV_n\rightarrow K(\Z_n,2).
\end{equation*}
On the other hand, recall the homotopy fiber sequence \eqref{eq:BPUn fib seq}:
\begin{equation*}
BU_n\rightarrow BPU_n\xrightarrow{\chi}K(\Z,3).
\end{equation*}
We compare the two homotopy sequences above by the following commutative (up to homotopy) diagram:
\begin{equation}\label{eq:compare}
\begin{tikzcd}
BV'_n\arrow[r]\arrow[d]&BV_n\arrow[r]\arrow[d]&K(\Z_n,2)\arrow[d,"\delta"]\\
BU_n\arrow[r]&BPU_n\arrow[r,"\chi"]&K(\Z,3).
\end{tikzcd}
\end{equation}
where the first two vertical arrows are induced by the inclusions of groups, and the third one $\delta$ is the Bockstein homomorphism.

Let $^VE_*^{*,*}$ and $^UE_*^{*,*}$ be the Serre spectral sequences for the upper and lower homotopy fiber sequences in \eqref{eq:compare}, respectively:
\begin{equation*}
\begin{split}
&^VE_2^{s,t}=H^s(K(\Z_n,2);H^t(BV'_n);\Z)\Rightarrow H^{s+t}(BV_n;\Z),\\
&^UE_2^{s,t}=H^s(K(\Z,3);H^t(BU_n);\Z)\Rightarrow H^{s+t}(BPU_n;\Z).
\end{split}
\end{equation*}
The only nontrivial group $^VE_2^{s,t}$ with $s+t=2$ is
\[^VE_2^{0,2}\cong H^2(BV'_n;\Z).\]
Therefore, by Lemma \ref{lem:H3}, we have
\begin{equation}\label{eq:VEinfty}
^VE_{\infty}^{0,2}\cong H^2(BV_n;\Z)\cong\Z_n\oplus\Z_n.
\end{equation}
By Lemma \ref{lem:V'}, we have a short exact sequence
\begin{equation}\label{eq:VE2}
0\to\Z_n\to {^VE}_2^{0,2}\to\Z_n\to 0.
\end{equation}
Comparing \eqref{eq:VEinfty} and \eqref{eq:VE2}, we have
\[^VE_{\infty}^{0,2}\cong{^VE}_2^{0,2}\cong\Z_n\oplus\Z_n\cong H^2(BV_n;\Z).\]
Hence, there is no nontrivial differential landing in ${^VE}_2^{3,0}\cong\Z_n$. By Lemma \ref{lem:H3}, we have
\[{^VE}_{\infty}^{3,0}={^VE}_2^{3,0}\cong H^3(BV_n;\Z)\cong\Z_n.\]
Then the diagram \eqref{eq:compare} induces an isomorphism
\[H^3(BPU_n;Z)={^UE}_{\infty}^{3,0}\rightarrow {^VE}_{\infty}^{3,0}={^VE}_2^{3,0}=H^3(BV_n;\Z),\]
and we conclude.
\end{proof}

Now we are ready to prove Theorem \ref{thm:main}:
\begin{theorem}[Theorem \ref{thm:main}]
Let $p$ be an odd prime. In dimensions $0\leq k \leq 2(p+1)$, the graded ring $\BP^*(BPU_n)\otimes_{BP^*}\Z_{(p)}$, or equivalently, the image of the Thom map $\BP^*(BPU_n)\to H^*(BPU_n;\Z_{(p)})$ concentrates in even dimensions.
\end{theorem}
\begin{proof}
It follows from Theorem \ref{thm:2p+2} that the $p$-torsion subgroup of $H^k(BPU_n;\mathbb{Z})$ is $0$ for $0<k<2p+2$, $k\neq3$. On the other hand, all the non-torsion classes in $H^*(BPU_n;\Z_{(p)})$ are of even dimensions. Therefore, it suffices to show that the image of the Thom map
\[T:\BP^3(BPU_n)\rightarrow H^3(BPU_n;\Z_{p})\]
is trivial.

Consider the commutative diagram
\begin{equation}\label{eq:Thom3}
\begin{tikzcd}
\BP^3(BPU_n)\arrow[r]\arrow[d,"T"]&\BP^3(BV_n)\arrow[d,"T"]\\
H^3(BPU_n;\Z_{(p)})\arrow[r,"\cong"]&H^3(BV_n;\Z_{(p)})
\end{tikzcd}
\end{equation}
where the vertical arrows are the Thom maps, and the horizontal ones are the restrictions. By Lemma \ref{lem:H3 restriction}, the bottom arrow is an isomorphism.

We complete the proof by contradiction. If the vertical arrow to the left has a nontrivial image, then so is the image of the composition
\begin{equation}\label{eq:composition}
\BP^3(BPU_n)\xrightarrow{T}H^3(BPU_n;\Z_{(p)})\xrightarrow{\cong} H^3(BV_n;\Z_{(p)}).
\end{equation}
However, by Landweber \cite{landweber1970coherence}, $\BP^*(BG)$ concentrates in even dimensions if $G$ is abelian, from which we deduce $\BP^3(BV_n)=0$. Therefore, by \eqref{eq:Thom3}, the composition \eqref{eq:composition} factors through $0$, leading to a contradiction.
\end{proof}

\section{The classes $\eta_{p,k}$}\label{sec:eta}
In this section we prove Theorem \ref{thm:eta p,k} by studying the Thom map for the Eilenberg-Mac Lane space $K(\Z,3)$. Indeed, the Thom maps for Eilenberg-Mac Lane spaces are studied in Tamanoi \cite{tamanoi1997image}, and we make use of one of his main conclusions. In what follows, we denote by $\mathscr{A}^*$ the mod $p$ Steenrod algebra for an odd prime $p$, and we use the notations for the stable cohomology operations in Milnor \cite{milnor1958steenrod}.
\begin{theorem}[Tamanoi, (I) of Theorem A, \cite{tamanoi1997image}]\label{thm:Tamanoi}
Let $p$ be a prime, and let $n\geq 1$. The image of the Thom map
\[T':\BP^*(K(\Z,n+2))\rightarrow H^*(K(\Z,n+2);\Z_{p})\]
is an $\mathscr{A}^*$-invariant polynomial subalgebra with infinitely many generators:
\[\opn{Im}T'=\Z_p[Q_{s_n}Q_{s_{n-1}}\cdots Q_{s_1}(\tau_{n+2})|0<s_1<\cdots<s_n],\]
where $\tau_{n+2}\in H^{n+2}(K(\Z,n+2);\Z_{p})$ is the fundamental class.
\end{theorem}
\begin{corollary}\label{cor:Tamonoi}
The classes $y_{p,k}\in H^*(K(\Z,3);\Z_{(p)})$ for $k\geq 0$ are in the image of the Thom map
\[T:\BP^*(K(\Z,3))\rightarrow H^*(K(\Z,3);\Z_{(p)}).\]
\end{corollary}
\begin{proof}
By definition, the Thom map $T'$ factors as
\begin{equation*}
\begin{tikzcd}
\BP^k(K(\Z,3))\arrow[dr,"T"]\arrow[rr,"T'"]& &H^k(K(\Z,3);\Z_p)\\
&H^k(K(\Z,3);\Z_{(p)})\arrow[ur,hook]
\end{tikzcd}
\end{equation*}
where the hooked arrow is monic for $k>3$, by Corollary \ref{cor:K(Z,3)p-tor}. The desired result then follows from Theorem \ref{thm:Tamanoi}.
\end{proof}
We proceed to prove Theorem \ref{thm:eta p,k}:
\begin{theorem}[Theorem \ref{thm:eta p,k}]
Let $p$ be an odd prime. For $k\geq 0$ and $p|n$, there are classes
\begin{equation*}
\eta_{p,k}\in\BP^{2p^{k+1}+2}(BPU_n)
\end{equation*}
satisfying
\[T(\eta_{p,k})=y_{p,k}\in H^{2p^{k+1}+2}(BPU_n;\Z_{(p)})\]
where $T$ is the Thom map.
\end{theorem}
\begin{proof}
It follows from Corollary \ref{cor:Tamonoi} and the commutative diagram
\begin{equation*}
\begin{tikzcd}
\BP^*(K(\Z,3))\arrow[r,"\chi^*"]\arrow[d,"T"]&\BP^*(BPU_n)\arrow[d,"T"]\\
H^*(K(\Z,3);\Z_{(p)})\arrow[r,"\chi^*"]&H^*(BPU_n;\Z_{(p)})
\end{tikzcd}
\end{equation*}
that the classes $y_{p,k}\in H^*(BPU_n;\Z_{(p)})$ are in the image of
\[T: \BP^*(BPU_n)\rightarrow H^*(K(\Z,3);\Z_{(p)}).\]
Let $\eta_{p,k}\in\BP^*(BPU_n)$ satisfy $T(\eta_{p,k})=y_{p,k}$, and we conclude.
\end{proof}
\bibliographystyle{abbrv}
\bibliography{BrownPetersonofBPUnrefV2}
\end{document}